\documentclass[a4paper,reqno]{amsart}

\usepackage{amsmath}
\usepackage{amsthm}
\usepackage{amssymb}
\usepackage{eucal}
\usepackage{needspace}
\usepackage{tikz-cd}
\usepackage[colorlinks]{hyperref}

\newcommand{\Spec}{\operatorname{Spec}}
\newcommand{\Max}{\operatorname{Max}}
\newcommand{\Sim}{\operatorname{Sim}}

\newcommand{\grade}{\operatorname{grade}}
\newcommand{\fdim}{\operatorname{fdim}}
\newcommand{\idim}{\operatorname{idim}}
\newcommand{\occ}{\operatorname{occ}}

\newcommand{\soc}{\operatorname{soc}}
\newcommand{\rad}{\operatorname{rad}}
\newcommand{\Tor}{\operatorname{Tor}}
\newcommand{\Ext}{\operatorname{Ext}}
\newcommand{\Hom}{\operatorname{Hom}}

\newcommand{\lto}{\longrightarrow}
\newcommand{\op}{\mathrm{op}}

\newtheorem{theorem}{Theorem}[section]
\newtheorem{lemma}[theorem]{Lemma}
\newtheorem{proposition}[theorem]{Proposition}
\newtheorem{corollary}[theorem]{Corollary}
\theoremstyle{definition}

\newtheorem{example}[theorem]{Example}
\theoremstyle{remark}
\newtheorem{remark}[theorem]{Remark}
\newtheorem{question}[theorem]{Question}

\numberwithin{equation}{section}

\title[Self-injective dimensions and Iyama grade dualities]
      {Self-injective dimensions and Iyama grade dualities over Noether algebras}
\author{Dawei Shen}
\address{School of Mathematics and Statistics, 
    Henan University, 
    Kaifeng, 
    Henan 475004, 
    P. R. China}
\email{sdw12345@mail.ustc.edu.cn}
\date{\today}

\begin{document}

\begin{abstract}
Let $A$ be an algebra module-finite over a commutative Noetherian ring. 
We first establish a local criterion for the $n$-Gorenstein condition 
in terms of the grades and flat dimensions of injective envelopes 
of simple modules over central localizations. 
As a main application, we give an affirmative answer, 
in this module-finite setting, 
to a question of Huang on the symmetry of finite self-injective dimension 
under the Auslander condition.
We also show that $A$ is Auslander--Gorenstein if and only if
it is Iwanaga--Gorenstein and has the fiberwise Iyama grade duality.
For Cohen--Macaulay algebras, central regular reductions to Artin algebras
allow both the Auslander condition and the grade duality to be tested
on simple modules over $A$ itself.
\end{abstract}

\maketitle

\section{Introduction}
Let $A$ be a two-sided Noetherian ring. 
For $n\geq 0$, recall that $A$ is $n$-Gorenstein if, 
in a minimal injective coresolution
\[
0\lto A\lto I^0(A)\lto I^1(A)\lto\cdots
\]
of $A$, 
one has $\fdim_A I^\ell(A)\leq \ell$ for every $0\leq \ell<n$.
The ring $A$ is said to satisfy the Auslander condition
if it is $n$-Gorenstein for every $n$. 
A two-sided Noetherian ring is called Iwanaga--Gorenstein
if it has finite self-injective dimension on both sides.
An Iwanaga--Gorenstein ring satisfying the Auslander condition
is said to be Auslander--Gorenstein.

The Auslander condition is a homological regularity condition 
on the minimal injective coresolution of the ring.  
It has been studied extensively in the representation theory of algebras and 
is closely related to the behavior of syzygies, grades, 
and self-injective dimensions.  
Beyond the Artinian setting, however, 
injective modules are typically much larger and the usual finite-dimensional 
dualities are no longer available. 

Algebras that are module-finite over commutative Noetherian rings 
form a natural class in which one can still combine noncommutative homological 
methods with commutative localization: central localization reduces questions 
to the semilocal algebras, 
while Kanda's description of indecomposable injectives relates 
the global injective theory to simple modules over these 
localizations \cite{Kanda}. 
This local structure is the basic mechanism used throughout the paper.

Recent work of Kl\'asz, Kleinau, Marczinzik and Marquardt gives a new
characterization of finite-dimensional  Auslander--Gorenstein algebras
\cite{KKMM}. They show that a finite-dimensional algebra is
Auslander--Gorenstein precisely when it has a well-defined Auslander--Reiten
bijection; equivalently, in the Iwanaga--Gorenstein case, the
Auslander--Gorenstein property is detected by a well-defined Iyama grade duality. 
The present paper develops Iyama's grade duality for algebras module-finite 
over commutative Noetherian rings.

Our first main result is a local criterion for the $n$-Gorenstein
property for module-finite  algebras over commutative Noetherian rings. 
Two ingredients lead to our local criterion. 
For an arbitrary two-sided Noetherian ring, 
the author characterized the $n$-Gorenstein condition 
in terms of the occurrence degrees and flat dimensions of indecomposable injective 
modules  \cite{Shen}. 
On the other hand, 
Kanda's description of indecomposable injective modules and Bass numbers 
in this module-finite setting  
identifies these occurrence degrees with the grades of simple modules 
over the central localizations \cite{Kanda}. 

Combining these results gives the following local criterion.

\begin{theorem}\label{thm:intro-pointwise}
Let $A$ be a module-finite algebra over a commutative Noetherian ring $R$.
For $n\geq0$, the following are equivalent.
\begin{enumerate}
\item $A$ is $n$-Gorenstein.
\item For every $\mathfrak p\in\Spec R$ and every simple $A_{\mathfrak p}$-module $S$,
\[
 \grade_{A_{\mathfrak p}}S
 \geq\min\{n,\fdim_{A_{\mathfrak p}}I_{A_{\mathfrak p}}(S)\}.
\]
\item For every such pair $(\mathfrak p,S)$ with $\grade_{A_{\mathfrak p}}S<n$,
\[
 \fdim_{A_{\mathfrak p}}I_{A_{\mathfrak p}}(S)
 =\grade_{A_{\mathfrak p}}S.
\]
\end{enumerate}
\end{theorem}

Taking all $n$ in Theorem~\ref{thm:intro-pointwise} gives a
characterization of the Auslander condition in terms of simple
modules over central localizations. 

\begin{corollary}\label{cor:intro-ac}
Let $A$ be a module-finite algebra over a commutative Noetherian ring $R$.
Consider the following conditions.
\begin{enumerate}
\item $A$ satisfies the Auslander condition.
\item For every $\mathfrak p\in\Spec R$ and every simple $A_{\mathfrak p}$-module $S$,
\[
 \fdim_{A_{\mathfrak p}}I_{A_{\mathfrak p}}(S)
 =\grade_{A_{\mathfrak p}}S.
\]
\item For every simple $A$-module $S$, one has $\fdim_A I_A(S)=\grade_A S$.

\end{enumerate}
Then $(1)\iff(2)\implies(3)$.
If $A$ is Cohen--Macaulay over $R$, then $(3)\implies(1)$.
\end{corollary}

The equivalence of $(1)$ and $(2)$ follows from the local criterion.
The Cohen--Macaulay case is proved in Section~\ref{sec:cm-reductions}
by reduction along a central regular sequence.

Our second main result concerns the symmetry of the finiteness of
self-injective dimensions. For Artin algebras satisfying the Auslander
condition, Auslander and Reiten proved that finiteness of 
the self-injective dimension on one side implies
finiteness on the other side \cite{AR}.
Huang extended this result to two-sided Artinian rings
\cite{Huang07}. For a general two-sided Noetherian ring, 
Huang asked whether the same implication holds \cite[Question~4.16]{Huang23}.
Zaks proved that for a two-sided Noetherian ring the two self-injective
dimensions are equal whenever both are finite \cite{Zaks}.

We prove this symmetry for
module-finite algebras over commutative Noetherian rings.
In contrast with the finite-dimensional characterization in \cite{KKMM},
our argument does not use a global vector-space duality or projective covers
of arbitrary injective modules.  Instead, it combines central localization,
Iwanaga's formula for self-injective dimension, and Iyama's grade duality on
the finitely many simple modules of each semilocal algebra.

\begin{theorem}\label{thm:intro-main}
Let $A$ be a module-finite algebra over a commutative Noetherian ring $R$. 
Assume that $A$ satisfies the Auslander condition. 
Then $\idim_AA$ is finite if and only if $\idim_{A^{\op}}A$ is finite. 
\end{theorem}

The argument for the second main result has one feature specific to the
module-finite situation over a commutative base ring.  
After localizing $A$ at a  prime $\mathfrak p$, 
the ring $A_{\mathfrak p}$ is semilocal.
Iyama's grade duality can then be applied degree by degree,
and a finite counting argument gives the  uniform bound \cite{Iyama03,Iyama05}.

Section~4 returns to Iyama's grade duality. 
Motivated  by the  characterization 
of finite-dimensional  Auslander--Gorenstein 
algebras in \cite{KKMM}, we consider a fiberwise form of Iyama's grade
duality for  module-finite algebras over commutative Noetherian rings.  
That section provides a local counterpart to the principle that
simple-module bijections detect the Auslander--Gorenstein property.

We say that $A$ has the Iyama grade duality
if every simple $A$-module has finite grade and
the assignment
\[
S\longmapsto
h_A(S):=
\soc_{A^{\op}}
\Ext_A^{\grade_A S}(S,A)
\]
is a well-defined bijection from $\Sim A$ to
$\Sim A^{\op}$.
We say that $A$ has the fiberwise Iyama grade duality 
if $A_{\mathfrak p}$ has the Iyama grade duality for every $\mathfrak p\in\Spec R$.

Combining the results of Section~4 with the Cohen--Macaulay 
reductions in Section~5 gives the following characterization.

\begin{theorem}\label{thm:intro-third}
Let $A$ be a module-finite algebra over a commutative Noetherian ring $R$.
Consider the following conditions.
  \begin{enumerate} 
  \item $A$ is Auslander--Gorenstein.
  \item $A$ is Iwanaga--Gorenstein and 
  $A$ has the fiberwise Iyama grade duality.
  \item $A$ is Iwanaga--Gorenstein and 
  $A$ has the Iyama grade duality.
  \end{enumerate}
  Then $(1)\iff (2) \implies (3)$. 
  If $A$ is Cohen--Macaulay over $R$, then $(3)\implies (1)$.
\end{theorem}

The local criterion is used to prove the symmetry of finite
self-injective dimension under the Auslander condition for
Noether algebras. A Tor detection argument gives the fiberwise
characterization, while central regular reductions show that
the fiberwise condition can be replaced by the ordinary Iyama
grade duality in the Cohen--Macaulay case.

The paper is organized as follows. In Section~2, we establish a local criterion for the
$n$-Gorenstein condition over algebras that are module-finite over commutative
Noetherian rings. Section~3 is devoted to the symmetry of self-injective
dimensions, where we prove the main theorem and derive its consequences.
In Section~4, we study Tor detection and the corresponding fiberwise grade
duality, which provide a local interpretation of the symmetry result and
establish the fiberwise Auslander--Gorenstein criterion. In Section~5, central regular reductions to Artin algebras give the Cohen--Macaulay cases of Corollary~\ref{cor:intro-ac} and Theorem~\ref{thm:intro-third}.

Throughout this paper, all modules are left modules unless otherwise stated.

\section{A local criterion for Auslander conditions}

Let $A$ be a two-sided Noetherian ring.
For a finitely generated $A$-module $M$, 
the grade of $M$ is defined by 
\[
\grade_A M=\inf\{i\geq 0\mid \Ext_A^i(M,A)\neq 0\}.
\]
The infimum of the empty set is $\infty$. 
Fix a minimal injective coresolution 
\[
0\lto A\lto I^0(A)\lto I^1(A)\lto\cdots
\]
of the regular module $A$. 
For an indecomposable injective $A$-module $I$, 
the occurrence degree of $I$ is defined by 
\[
\occ_AI=\inf\{i\geq 0\mid I\text{ is a direct summand of }I^i(A)\}.
\]
For an $A$-module $M$, we denote by $I_A(M)$ an injective envelope of $M$.

We need the following lemma; see \cite[Lemmas 2.1 and 2.2]{Shen}.
\begin{lemma}\label{lem:grade-fd}
Let $S$ be a simple $A$-module. Then
\[
    \grade_AS=\occ_AI_A(S)\leq \fdim_AI_A(S).
\]
\end{lemma}

For $n\geq 0$, 
recall that $A$ is $n$-Gorenstein if the flat dimension of $I^\ell(A)$ 
is at most $\ell$ for every $0\leq \ell<n$.
The ring $A$ is said to satisfy the Auslander condition
if it is $n$-Gorenstein for every $n\geq 0$, and 
$A$ is Auslander--Gorenstein if it is Iwanaga--Gorenstein and 
satisfies the Auslander condition.

Let $M$ be a finitely generated $A$-module. The strong grade of $M$
is defined by
\[
\operatorname{s.grade}_A M
=\inf\{\grade_A N\mid N\subseteq M\}.
\]

For each $\ell\geq0$, the finitely generated 
$A$-modules of strong grade at least $\ell$
form a Serre class. By minimality of the injective coresolution, 
this is equivalent to
\[\Hom_A(M,I^j(A))=0 \text{ for } 0\leq j <\ell.\]
The assertion follows from the injectivity of the modules $I^j(A)$.

We need the following; 
see \cite[Theorem 3.7]{FGR} and \cite[Theorem 3.3]{Shen}.

\begin{lemma}\label{lem:n-gor-criteria}
For $n\geq 0$, the following are equivalent.
\begin{enumerate}
   \item $A$ is $n$-Gorenstein;
   \item for every finitely generated $A$-module $M$ and every $1\leq \ell\leq n$, 
   \[\operatorname{s.grade}_{A^{\op}}\Ext^\ell_A(M,A)\geq \ell;\]
   \item for every indecomposable injective $A$-module $I$ with $\occ_AI<n$,
   \[\occ_AI\geq \fdim_AI;\]
   \item $A^{\op}$ is $n$-Gorenstein.
\end{enumerate}
\end{lemma}

We now restrict our attention to Noether algebras.
Let $R$ be a commutative Noetherian ring.
An $R$-algebra $A$ that is finitely generated as an $R$-module
is called a Noether algebra.
We write $\Spec A$ for the set of prime ideals of $A$.

By \cite[Theorem 7.1]{Kanda}, 
there is a bijection between $\Spec A$ and the isomorphism classes of 
indecomposable injective $A$-modules, under which $P$ corresponds to $I_P$. 
More precisely, for each prime ideal $P$ of $A$, the injective envelope of $A/P$ is 
a finite direct sum of copies of $I_P$. Conversely, 
every indecomposable injective $A$-module is isomorphic to $I_P$ 
for a unique prime $P$ of $A$.

We need the following lemma; see \cite[\S2.4--2.5]{KN}.

\begin{lemma}\label{lem:plocal-ptorsion}
Let $P\in\Spec A$ and put $\mathfrak p=P\cap R$.
Then multiplication by every $c\in R\setminus\mathfrak p$
is an automorphism of $I_P$, 
and every element of $I_P$ is annihilated by some power
of $\mathfrak p$.
Consequently, $I_P$ is $\mathfrak p$-local and $\mathfrak p$-torsion.
\end{lemma}

For $P\in\Spec A$, put $\mathfrak p=P\cap R$. 
Then $PA_{\mathfrak p}$ is a maximal  ideal of $A_{\mathfrak p}$, 
and the quotient $A_{\mathfrak p}/PA_{\mathfrak p}$ is a simple Artinian ring.
There is, up to isomorphism, 
a unique simple $A_{\mathfrak p}$-module annihilated by $PA_{\mathfrak p}$; 
denote it by $S_P$. Conversely, 
every simple $A_{\mathfrak p}$-module is isomorphic to $S_P$ for a unique prime $P$ 
with $P\cap R=\mathfrak p$.

We recall the following result from \cite{Kanda}.
For a ring $A$, let $\Sim A$ denote
the set of isomorphism classes of  simple $A$-modules.

\begin{proposition}\label{prop:dictionary}
There is a bijection
\[\begin{aligned}
\varphi\colon\Spec A &\xrightarrow{\simeq}
\coprod_{\mathfrak p\in\Spec R}\Sim A_{\mathfrak p}\\
P &\longmapsto (P\cap R,S_P).
\end{aligned}\]
If $\varphi(P)=(\mathfrak p,S)$, then the following statements hold.
\begin{enumerate}
   \item $I_P\cong I_{A_{\mathfrak p}}(S)$ as an $A_{\mathfrak p}$-module;
   \item $\fdim_A I_P=\fdim_{A_{\mathfrak p}}I_{A_{\mathfrak p}}(S)$;
   \item $\occ_A I_P=\grade_{A_{\mathfrak p}}S$.
\end{enumerate}
\end{proposition}

\begin{proof}
The bijection follows from \cite[Theorem 7.6]{Kanda}.
The proof of \cite[Proposition 7.8]{Kanda}  shows that
$I_P$, regarded as an $A_{\mathfrak p}$-module, is isomorphic
to $I_{A_{\mathfrak p}}(S)$.

For every $\mathfrak p$-local $A$-module $M$,
$\fdim_A M=\fdim_{A_{\mathfrak p}}M$. 
Indeed, localization gives 
\[\fdim_{A_{\mathfrak p}}M\leq\fdim_A M.\] 
Since $A_{\mathfrak p}$ is $A$-flat, restriction of scalars gives
the reverse inequality.
By Lemma~\ref{lem:plocal-ptorsion}, the module $I_P$ is $\mathfrak p$-local, 
so applying this observation to $M=I_P$ yields the flat-dimension equality.

Finally, by \cite[Theorem 7.10]{Kanda}, the multiplicity of $I_P$ in $I^i(A)$
is nonzero if and only if $\Ext^i_{A_{\mathfrak p}}(S,A_{\mathfrak p})$
is nonzero.
Taking the first such $i$ gives the last equality.
\end{proof}

We have the following local criterion.

\begin{theorem}\label{thm:pointwise}
Let $A$ be a module-finite algebra over a commutative Noetherian ring $R$.
For $n\geq 0$, the following are equivalent.
\begin{enumerate}
\item $A$ is $n$-Gorenstein.
\item For every $\mathfrak p\in\Spec R$ and every simple $A_{\mathfrak p}$-module $S$,
\[
\grade_{A_{\mathfrak p}}S
\geq
\min\{n,\fdim_{A_{\mathfrak p}}I_{A_{\mathfrak p}}(S)\}.
\]
\item For every such pair $(\mathfrak p,S)$ with $\grade_{A_{\mathfrak p}}S<n$,
\[
\fdim_{A_{\mathfrak p}}I_{A_{\mathfrak p}}(S)
=
\grade_{A_{\mathfrak p}}S.
\]
\end{enumerate}
\end{theorem}

\begin{proof}
The equivalence (1)$\iff$(2) follows from 
Lemma~\ref{lem:n-gor-criteria} 
and Proposition~\ref{prop:dictionary}.
The equivalence (2)$\iff$(3) follows from Lemma~\ref{lem:grade-fd}.
\end{proof}

\begin{corollary}\label{cor:grade=fd}
The following are equivalent.
\begin{enumerate}
  \item $A$ satisfies the Auslander condition;
  \item $\fdim_{A_{\mathfrak p}}I_{A_{\mathfrak p}}(S)=\grade_{A_{\mathfrak p}}S$ 
  for every pair $(\mathfrak p,S)$.
\end{enumerate}
\end{corollary}

\begin{proof}
This follows from Theorem~\ref{thm:pointwise} by applying condition~(3) for
all $n$, together with Lemma~\ref{lem:grade-fd} when the grade is infinite.
\end{proof}

\begin{corollary}\label{cor:localglobal}
For $n\geq 0$, the following are equivalent.
\begin{enumerate}
\item $A$ is $n$-Gorenstein;
\item $A_{\mathfrak p}$ is $n$-Gorenstein for every $\mathfrak p\in\Spec R$;
\item $A_{\mathfrak m}$ is $n$-Gorenstein for every $\mathfrak m\in\Max R$.
\end{enumerate}
\end{corollary}

\begin{proof}
(1)$\implies$(2) 
Fix $\mathfrak p\in\Spec R$. 
The primes of $R_{\mathfrak p}$ are $\mathfrak qR_{\mathfrak p}$ 
with $\mathfrak q\subseteq\mathfrak p$. 
Since $(A_{\mathfrak p})_{\mathfrak qR_{\mathfrak p}}$ and $A_{\mathfrak q}$ 
are isomorphic, the inequalities for $A_{\mathfrak p}$ are among those for $A$. 

(2)$\implies$(3) 
This is immediate.

(3)$\implies$(1)
Given $\mathfrak p\in\Spec R$, 
choose $\mathfrak m\in\Max R$ with $\mathfrak p\subseteq\mathfrak m$. 
The algebra $A_{\mathfrak m}$ is $n$-Gorenstein. 
Since $(A_{\mathfrak m})_{\mathfrak pR_{\mathfrak m}}$ and $A_{\mathfrak p}$ 
are isomorphic, $A_{\mathfrak p}$  is $n$-Gorenstein. 
Then $A_{\mathfrak p}$ is $n$-Gorenstein for every $\mathfrak p\in\Spec R$. 
By Theorem~\ref{thm:pointwise}, 
the required inequality holds for every pair $(\mathfrak p,S)$. 
Therefore, $A$ is $n$-Gorenstein.
\end{proof}

We recall the following result; see 
\cite[Proposition~1]{Iwanaga} and \cite[Lemma~3.7]{HH}.

\begin{lemma}\label{lem:iwanaga}
Let $A$ be a two-sided Noetherian ring. Then
\[\begin{aligned}
\idim_{A^{\op}}A
&=\sup\{\fdim_{A} I\mid I \text{ an injective $A$-module}\}\\
&=\sup\{\fdim_{A} I_{A}(S)\mid S \text{ a simple $A$-module}\}.
\end{aligned}\]
\end{lemma}

By \cite[Lemma~3.7]{HH}, $\idim_{A^{\op}}A=\fdim_A Q$,
where $Q$ is the direct sum of the injective envelopes of all simple $A$-modules.
The second equality follows from
\[
\fdim_A Q=\sup_{S\in\Sim A}\fdim_A I_A(S).
\]

The local criterion has an application to Frobenius algebras over
commutative Gorenstein rings. Here a commutative Noetherian ring $R$ is called
\emph{Gorenstein} if $R_{\mathfrak p}$ is a Gorenstein local ring for every
$\mathfrak p\in\Spec R$. A \emph{Frobenius $R$-algebra} means a finite
projective $R$-algebra $A$ for which
\[
             {}_A A_R\cong {}_A\Hom_R(A,R)_R
\]
as $A$-$R$-bimodules, where the $A$-action on $\Hom_R(A,R)$ is given by
\[(a\cdot f)(b)=f(ba).\]

\begin{proposition}\label{prop:frobenius-gorenstein-base}
Let $R$ be a commutative Gorenstein ring and let $A$ be a
Frobenius $R$-algebra. For every $\mathfrak p\in\Spec R$ and every simple
$A_{\mathfrak p}$-module $S$, 
\[
 \fdim_{A_{\mathfrak p}}I_{A_{\mathfrak p}}(S)
 =\grade_{A_{\mathfrak p}}S
 =\operatorname{ht}\mathfrak p<\infty.
\]
In particular, $A$ satisfies the Auslander condition.
\end{proposition}

\begin{proof}
Fix a prime $\mathfrak p\in\Spec R$. 
Put $h=\operatorname{ht}\mathfrak p$ and let
$S$ be a simple $A_{\mathfrak p}$-module. The Frobenius structure localizes,
so there is an isomorphism of $A_\mathfrak{p}$-modules
\[
 \Hom_{R_{\mathfrak p}}(A_{\mathfrak p},R_{\mathfrak p})
 \cong A_{\mathfrak p}.
\]
Since $A_{\mathfrak p}$ is projective over
$R_{\mathfrak p}$, an $A_{\mathfrak p}$-projective resolution is also
$R_{\mathfrak p}$-projective, and adjunction gives an isomorphism
\[
 \Ext^i_{A_{\mathfrak p}}(S,A_{\mathfrak p})
 \cong
 \Ext^i_{R_{\mathfrak p}}(S,R_{\mathfrak p})
\]
for $i\geq 0$. 
Since $S$ is finitely generated over $A_{\mathfrak p}$, 
it is finitely generated over $R_{\mathfrak p}$. 
By Nakayama's Lemma, the maximal ideal $\mathfrak pR_{\mathfrak p}$ annihilates $S$.
Then $S$ has finite length over
$R_{\mathfrak p}$. 
Since the Gorenstein local ring $R_{\mathfrak p}$ is
Cohen--Macaulay of dimension $h$, every nonzero finite-length
$R_{\mathfrak p}$-module has grade $h$. Hence,
\[
                 \grade_{A_{\mathfrak p}}S=h.
\]

Let $E=I_{R_{\mathfrak p}}(k(\mathfrak p))$. Lemma~\ref{lem:grade-fd}
and Lemma~\ref{lem:iwanaga} give
\[
 h=\grade_{R_{\mathfrak p}}k(\mathfrak p)
 \leq \fdim_{R_{\mathfrak p}}E
 \leq \idim_{R_{\mathfrak p}}R_{\mathfrak p}=h,
\]
and hence $\fdim_{R_{\mathfrak p}}E=h$. 
Put $J=\Hom_{R_{\mathfrak p}}(A_{\mathfrak p},E)$.
By adjunction, $J$ is an injective $A_{\mathfrak p}$-module. Since
$A_{\mathfrak p}$ is finite projective over $R_{\mathfrak p}$, the Frobenius
isomorphism gives
\[
 J\cong
 \Hom_{R_{\mathfrak p}}(A_{\mathfrak p},R_{\mathfrak p})
       \otimes_{R_{\mathfrak p}}E
 \cong A_{\mathfrak p}\otimes_{R_{\mathfrak p}}E,
\]
so $\fdim_{A_{\mathfrak p}}J\leq h$. 
Moreover, we have 
\[
 \Hom_{A_{\mathfrak p}}(S,J)
 \cong \Hom_{R_{\mathfrak p}}(S,E)\neq0.
\]
Then $S$ embeds in $J$. 
Extend this embedding to the injective envelope $I_{A_{\mathfrak p}}(S)$ of $S$.
The resulting map
\[I_{A_{\mathfrak p}}(S)\lto J\] 
is injective, since its kernel has zero
intersection with $S$, and it splits because $I_{A_{\mathfrak p}}(S)$ is
injective. Hence, $I_{A_{\mathfrak p}}(S)$ is a direct summand of $J$ and
\[
 \fdim_{A_{\mathfrak p}}I_{A_{\mathfrak p}}(S)\leq h.
\]
Lemma~\ref{lem:grade-fd} gives the reverse inequality, and consequently
\[
 \fdim_{A_{\mathfrak p}}I_{A_{\mathfrak p}}(S)
 =\grade_{A_{\mathfrak p}}S=h.
\]
Corollary~\ref{cor:grade=fd} now shows that $A$ satisfies the Auslander
condition.
\end{proof}

\begin{remark}
Ascent of the Auslander--Gorenstein property along Frobenius extensions has
been studied in several settings; see \cite{HKK}. 
Proposition~\ref{prop:frobenius-gorenstein-base} gives
\[
 \fdim_{A_{\mathfrak p}}I_{A_{\mathfrak p}}(S)
 =\grade_{A_{\mathfrak p}}S
 =\operatorname{ht}\mathfrak p<\infty.
\]
\end{remark}

\begin{example}
Let $\ell\geq 1$ and 
$A=k_{\mathbf q}[x_1,\ldots,x_m]$ be a quantum affine space with
\[q_{ij}^{\ell}=1 \text{ for all } 1\leq i,j\leq m.\] 
Then $A$ is a free Frobenius algebra over the 
commutative subalgebra
\[
             R=k[x_1^\ell,\ldots,x_m^\ell];
\]
see \cite[Proposition~3.7]{LT}. Since $R$ is Gorenstein,
Proposition~\ref{prop:frobenius-gorenstein-base} applies. 
Then
\[
 \fdim_{A_{\mathfrak p}}I_{A_{\mathfrak p}}(S)
 =\grade_{A_{\mathfrak p}}S
 =\operatorname{ht}\mathfrak p<\infty
\]
for every $\mathfrak p\in\Spec R$ and every
simple $A_{\mathfrak p}$-module $S$.
\end{example}

\section{Symmetry of self-injective dimensions}

For finite-dimensional algebras, the passage from the Auslander condition
together with finite self-injective dimension on one side to the
Iwanaga--Gorenstein property goes back to Auslander--Reiten.  
For an algebra module-finite over a commutative Noetherian ring, the injective
modules involved need not be finitely generated, so the finite-dimensional
Auslander--Reiten bijection is not available in this form.  The substitute
used here is local: Iwanaga's formula converts a bound on self-injective
dimension into flat-dimension bounds for injectives, while Iyama's grade
duality transfers the resulting grade bounds between the simple modules of
$A_{\mathfrak p}$ and $A_{\mathfrak p}^{\op}$.  The finiteness of
the simple spectra of these semilocal algebras is what makes the transfer
uniform.

For a ring $A$ and an integer $n>0$, 
denote by $\Sim_{<n} A$ the isomorphism classes of
simple $A$-modules of grade less than $n$.
We  recall the grade duality that will be used
in the counting argument.

\begin{lemma}\label{lem:iyama}
Let $A$ be a module-finite algebra over a commutative Noetherian ring $R$. 
Assume that $A$ is $n$-Gorenstein.
For every $\mathfrak{p}\in\Spec R$, the assignment
\[
S\longmapsto
\soc_{A_{\mathfrak p}^{\op}}
\Ext_{A_{\mathfrak p}}^{\grade_{A_{\mathfrak p}}S}
(S,A_{\mathfrak p})
\]
is a grade-preserving bijection from $\Sim_{<n}(A_{\mathfrak p})$
to $\Sim_{<n}(A_{\mathfrak p}^{\op})$. 
\end{lemma}

\begin{proof}
By Corollary \ref{cor:localglobal}, 
the algebra $A_{\mathfrak p}$ is also $n$-Gorenstein.
Then the conclusion follows from \cite[Theorem 1.3]{Iyama03}; 
see also \cite[\S3.6.3]{Iyama05}.
\end{proof}

We have the following.

\begin{corollary}\label{cor:counting}
Let $(R,\mathfrak m)$ be a commutative Noetherian local ring and 
let $A$ be a module-finite $R$-algebra. 
If $A$ is $n$-Gorenstein, the following are equivalent.
\begin{enumerate}
\item $\grade_A S< n$ for every simple $A$-module $S$;
\item $\grade_{A^{\op}}T<n$ for every simple $A^{\op}$-module $T$.
\end{enumerate}
\end{corollary}

\begin{proof}
Let $S$ be a simple $A$-module. 
Since $S$ is cyclic over $A$ and $A$ is finite over $R$, 
the module $S$ is finitely generated over $R$. 
The submodule $\mathfrak mS$ is an $A$-submodule, 
so Nakayama's lemma gives $\mathfrak mS=0$. 
The same argument applies to simple $A^{\op}$-modules.

The simple $A$-modules are precisely the simple
$A/\mathfrak mA$-modules. Since
$A/\mathfrak mA$ is a finite-dimensional $R/\mathfrak m$-algebra, there are finitely
many isomorphism classes on both sides, and $A/\mathfrak mA$ and
$(A/\mathfrak mA)^{\op}$ have the same number of simple-module
isomorphism classes.

Since $A$ is $n$-Gorenstein,
Lemma \ref{lem:iyama} gives a bijection between 
the simple $A$-modules of grade less than $n$ and 
the simple $A^{\op}$-modules of grade less than $n$. 
Suppose (2) holds. 
All simple $A^{\op}$-modules have grade less than $n$.
Since the total numbers of simple $A$-modules 
and simple $A^{\op}$-modules are equal, 
the $A$-modules of grade less than $n$ exhaust all simple $A$-modules. 
This proves (1).
The converse follows by interchanging $A$ and $A^{\op}$.
\end{proof}

We can now state the second main theorem in the form used in this section.

\begin{theorem}\label{thm:main}
Let $A$ be a module-finite algebra over a commutative Noetherian ring $R$. 
Assume that $A$ satisfies the Auslander condition. 
Then $\idim_AA$ is finite if and only if $\idim_{A^{\op}}A$ is finite. 
In this case, the two dimensions are equal.
\end{theorem}

\begin{proof}
Assume that $d=\idim_A A$ is finite.
Fix $\mathfrak p\in\Spec R$ and 
let $T$ be a simple $A_{\mathfrak p}^{\op}$-module. 
Let $J$ be its injective envelope.
Then $J$ is injective as an
$A^{\op}$-module. 

Lemma \ref{lem:iwanaga} gives $\fdim_{A^{\op}}J\leq d$. 
Since $J$ is $\mathfrak p$-local, 
the proof of Proposition \ref{prop:dictionary} gives
\[
\fdim_{A_{\mathfrak p}^{\op}}J
=
\fdim_{A^{\op}}J
\leq d.
\]
By Corollary~\ref{cor:localglobal}, $A_{\mathfrak p}$ satisfies the
Auslander condition, and Lemma~\ref{lem:n-gor-criteria} shows that
$A_{\mathfrak p}^{\op}$ does as well. Corollary~\ref{cor:grade=fd},
applied to $A_{\mathfrak p}^{\op}$, gives
\[
\grade_{A_{\mathfrak p}^{\op}}T
=
\fdim_{A_{\mathfrak p}^{\op}}J
\leq d.
\]
Then every simple $A_{\mathfrak p}^{\op}$-module has grade at most $d$. 
Since $A_{\mathfrak p}$ is $(d+1)$-Gorenstein, Corollary~\ref{cor:counting}
shows that every simple $A_{\mathfrak p}$-module has grade at most $d$.
Now let $I_P$ be an arbitrary indecomposable injective $A$-module 
and $\varphi(P)=(\mathfrak p,S)$ under Proposition \ref{prop:dictionary}. 
Then, by Corollary~\ref{cor:grade=fd},
\[
\fdim_A I_P
=
\fdim_{A_{\mathfrak p}}I_{A_{\mathfrak p}}(S)
=
\grade_{A_{\mathfrak p}}S
\leq d.
\]

Since $A$ is module-finite over the Noetherian ring $R$, 
every injective $A$-module is a direct sum of indecomposable injective modules \cite{FW}. 
Direct sums of modules of flat dimension at most $d$ again have flat dimension at most $d$. 
Then every injective $A$-module $I$ satisfies $\fdim_A I\leq d$. 

Applying Lemma~\ref{lem:iwanaga} gives 
\[
\idim_{A^{\op}}A\leq d<\infty.
\]
Applying the same argument to $A^{\op}$ yields the reverse
inequality, and hence
\[
\idim_AA=\idim_{A^{\op}}A;
\]
see also \cite{Zaks}. This completes the proof.
\end{proof}

\begin{corollary}
Let $A$ be a module-finite algebra over a commutative Noetherian ring. 
The following are equivalent.
\begin{enumerate}
    \item $A$ is Auslander--Gorenstein;
    \item $d=\idim_AA<\infty$ and $A$ is $d$-Gorenstein;
    \item $d=\idim_{A^{\op}}A<\infty$ and $A$ is $d$-Gorenstein.
\end{enumerate}
\end{corollary}

\begin{proof}
(1)$\implies$(2) This is immediate. 

(2)$\implies$(1)  
If $d=0$, then $A$ is self-injective and hence satisfies the Auslander condition. 
Let $d>0$ and consider the finite injective coresolution
\[
0\lto A\lto I^0(A)\lto\cdots\lto I^d(A)\lto 0.
\]
Since $A$ is $d$-Gorenstein, 
$\fdim_A I^i(A)\leq i$ for $i<d$. 
For $0\leq i\leq d$, let $K^i$ be the kernel of $I^{i}(A)\to I^{i+1}(A)$. 
From the short exact sequences
\[
0\lto K^{i}\lto I^i(A)\lto K^{i+1}\lto0
\]
an induction gives $\fdim_A K^i\leq i$. The final exact sequence
\[
0\lto K^{d-1}\lto I^{d-1}(A)\lto I^d(A)\lto0
\]
then gives $\fdim_A I^d(A)\leq d$. 
Then every term of the minimal injective coresolution has flat dimension 
at most its degree, so $A$ satisfies the Auslander condition. 
Theorem~\ref{thm:main} then gives $\idim_{A^{\op}}A<\infty$, 
and hence $A$ is Auslander--Gorenstein. 

(3)$\iff$(1) This follows by applying the same argument to $A^{\op}$ 
and  using Lemma~\ref{lem:n-gor-criteria}.
\end{proof}

\section{Iyama grade duality}
Kl\'asz et al. show that, for a finite-dimensional Iwanaga--Gorenstein algebra, 
the Auslander--Gorenstein
property can be characterized by a bijective grade map on simple modules
\cite[Theorem 2.1]{KKMM}.  A key homological step in their proof relates the
top projective degree of an injective module to the grade of a simple module
through finite-dimensional duality \cite[Lemma 2.2]{KKMM}.  

The purpose of this section is to isolate an analog suited to the 
present module-finite setting. 
Rather than dualizing over a field, we pair an injective $A$-module with
$A^{\op}$-modules by Tor. 
For injective envelopes over a central
localization, which are torsion over the corresponding local ring of the base, 
the top Tor degree is detected by the semisimple quotient.  This leads naturally
to a fiberwise form of Iyama's grade duality on simple modules.

\begin{lemma}\label{lem:tor-grade}
Let $M$ be a finitely generated $A$-module of finite 
grade $g$, and let $T$ be a submodule of $\Ext_A^g(M,A)$. 
If $I$ is an injective $A$-module with finite flat  
dimension $d$ and $\Tor_d^A(T,I)$ is nonzero,
then $d=g$ and there is an induced embedding
\[\Tor_d^A(T,I)\rightarrowtail \Hom_A(M,I).\]
In particular, $\Hom_A(M,I)$ is nonzero.
\end{lemma}

\begin{proof}
Choose a finitely generated projective resolution $P$ of $M$.
Put $P_{-1}^*=0$ and $P_i^*=\Hom_A(P_i,A)$ for $i\geq 0$.
Since $g$ is the first nonzero cohomological degree of $P^*$, 
the cokernel $C$ of $P_{g-1}^*\to P_g^*$ has a projective resolution
\[
0\lto P_0^*\lto P_1^*\lto\cdots
\lto P_g^*\lto C\lto0
\]
and $\Ext_A^g(M,A)$ embeds in $C$.
Then $\operatorname{pdim}_{A^{\op}}C\leq g$. 
The evaluation isomorphisms between $P_j^*\otimes_A X$ 
and $\Hom_A(P_j,X)$ 
give a natural isomorphism
\[
 \Tor_i^A(C,X)\cong\Ext_A^{g-i}(M,X)
\]
for every $A$-module $X$ and every $1\leq i\leq g$. 
If $X$ is flat, they also give
\[
             \Ext_A^g(M,A)\otimes_A X\cong\Ext_A^g(M,X).
\]

Since $T$ embeds in $C$ and $\fdim_A I=d$, 
there is an embedding
\[
\Tor_d^A(T,I)\rightarrowtail \Tor_d^A(C,I).
\] 
Since $\operatorname{pdim}_{A^{\op}}C\leq g$ and
$\Tor_d^A(C,I)\neq0$, we have $d\leq g$.
If $d\geq1$, the Tor--Ext isomorphism above gives
\[
0\neq \Tor_d^A(T,I)\rightarrowtail \Tor_d^A(C,I)\cong\Ext_A^{g-d}(M,I).
\]
Since $I$ is injective, $g=d$.

If $d=0$, then $I$ is flat. 
The inclusion $T\subseteq\Ext_A^g(M,A)$ remains injective after tensoring with $I$, 
so the flat evaluation isomorphism yields
\[
0\neq T\otimes_A I\rightarrowtail \Ext_A^g(M,A)\otimes_A I
   \cong\Ext_A^g(M,I).
\]
Injectivity of $I$ forces $g=d=0$.
\end{proof}

Let $(R,\mathfrak m)$ be a commutative Noetherian local ring
and let $A$ be a module-finite $R$-algebra.
An $A$-module $I$ is called $\mathfrak{m}$-torsion 
if for every $x\in I$, there exists an integer $k$ such that $\mathfrak{m}^kx=0$.

The following is related to \cite[Proposition~3.1]{NP} 
and \cite[Proposition~5.5]{AF}.

\begin{lemma}\label{lem:top-tor-central}
Let $M$ be an $A^{\op}$-module and 
$I$ be an $\mathfrak m$-torsion $A$-module. If
\[
  \Tor_n^A(M,I)\neq0,
\]
then there exist a simple $A^{\op}$-module $T$ and an integer $m\geq n$ 
such that
\[
  \Tor_m^A(T,I)\neq0.
\]
\end{lemma}

\begin{proof}
Write $M$ as the filtered colimit of 
its finitely generated submodules $\{M_i\}$. 
Since Tor commutes with filtered colimits,
$\Tor_n^A(M_i,I)$ is nonzero for some index $i$.
Replacing $M$ by $M_i$, we may assume that $M$ is
finitely generated.

The class of $\mathfrak{m}$-torsion modules is closed under 
submodules, quotients, extensions, and arbitrary direct sums.
Since $I$ is $\mathfrak m$-torsion,
using a free resolution of $M$, 
we see that $\Tor_i^A(M,I)$ is $\mathfrak m$-torsion for every $i\geq0$.

Let $x\in\mathfrak m$. 
Put $K=0:_M x \text{ and } C=M/xM$.
The exact sequences
\[\begin{gathered}
0\lto K\lto M\lto xM\lto0\\
0\lto xM\lto M\lto C\lto0
\end{gathered}\]
show that if $\Tor_n^A(K,I)=0=\Tor_{n+1}^A(C,I)$, then the induced maps
\[
\Tor_n^A(M,I)\lto\Tor_n^A(xM,I)
\lto\Tor_n^A(M,I)
\]
are both injective. By functoriality of Tor, 
their composite is multiplication by $x$ on $\Tor_n^A(M,I)$. 
However, $\Tor_n^A(M,I)$ is $\mathfrak m$-torsion.
Then every nonzero element is annihilated by some power of $x$,
contradicting injectivity. Therefore,
\[
\Tor_n^A(K,I)\neq0 \text{ or }
\Tor_{n+1}^A(C,I)\neq0.
\]
Both $K$ and $C$ are annihilated by $x$.

Let $x_1,x_2,\ldots,x_d$ be a system of parameters of $R$.
Apply this step successively to the parameters $x_1,x_2,\ldots,x_d$. 
At each stage the annihilation obtained in the preceding stages is preserved 
under choosing the relevant submodule or quotient. 
We obtain a finitely generated $A^{\op}$-module $L$ with $x_iL=0$ 
for all $1\leq i\leq d$ 
and $\Tor_m^A(L,I)\neq0$ for some $n\leq m\leq n+d$,
since at each of the $d$ steps the Tor degree either stays unchanged 
or increases by one. 

Since $A$ is module-finite over $R$, 
the module $L$ is finitely generated over $R$. 
Since $x_iL=0$ for all $1\leq i\leq d$, the module $L$ has finite length. 
A composition series of $L$ yields a simple factor $T$ with $\Tor_m^A(T,I)\neq0$.
\end{proof}

\begin{corollary}\label{cor:residue-tor-short}
Let $I$ be a nonzero $\mathfrak m$-torsion $A$-module. Then
\[
\fdim_AI=\sup\{i\geq0\mid\Tor_i^A(A/\rad A,I)\neq0\}.
\]
\end{corollary}

\begin{proof}
Put $s=\sup\{i\geq0\mid\Tor_i^A(A/\rad A,I)\neq0\}$. 
We have $s\leq\fdim_A I$. Conversely, 
let $M$ be an $A^{\op}$-module such that $\Tor_n^A(M,I)$ is nonzero. 
Lemma~\ref{lem:top-tor-central} gives a simple $A^{\op}$-module $T$ 
and an integer $m\geq n$ such that $\Tor_m^A(T,I)$ is nonzero. 
Since $A$ is semilocal, $A/\rad A$ is semisimple Artinian and $T$ occurs as 
a direct summand. Then $\Tor_m^A(A/\rad A,I)$ is nonzero, so $n\leq m\leq s$. 
Taking the supremum over all such $M$ and $n$ gives $\fdim_A I\leq s$.
\end{proof}

Motivated by Lemma~\ref{lem:iyama}, we say that a module-finite algebra $A$
over a commutative Noetherian ring $R$  has the \emph{Iyama grade duality}
if every simple $A$-module has finite grade and
the assignment
\[
S\longmapsto
h_A(S):=
\soc_{A^{\op}}
\Ext_A^{\grade_A S}(S,A)
\]
is a well-defined bijection from $\Sim A$ to
$\Sim A^{\op}$.
We say that $A$ has the \emph{fiberwise Iyama grade
duality} if $A_{\mathfrak p}$ has the Iyama grade
duality for every $\mathfrak p\in\Spec R$.

The fiberwise Iyama grade duality is a stronger
condition than the Iyama grade duality for the algebra itself.
Note that $A$ has the Iyama grade duality if and
only if $A_{\mathfrak m}$ has it for every
$\mathfrak m\in\Max R$.

The following example distinguishes the ordinary and fiberwise
Iyama grade dualities.

\begin{example}\label{ex:noncm-grade-bijection}
Let $k$ be a field and put
\[
 A=k[[t,u,v,z]]/(u^2,uv,v^2,uz,vz,z^2,tz).
\]
We regard $A$ as an algebra over itself. It is a commutative
Noetherian local ring with maximal ideal $\mathfrak n=(t,u,v,z)$
and unique simple module $S=A/\mathfrak n$.
Every element of $A$ has a unique expression
\[
 a(t)+b(t)u+c(t)v+\lambda z, \text { where } 
 a(t), b(t), c(t)\in k[[t]],\lambda\in k.
\]
Multiplication by $t$ shows that every element annihilated by
$\mathfrak n$ belongs to $kz$, and $\mathfrak n z=0$. 
Hence $\soc_A A=kz$ and $\grade_A S=0$.
It follows that
\[
 h_A(S)=\Hom_A(S,A)\cong kz\cong S.
\]
Then $A$ has the Iyama grade duality.

The ideal $\mathfrak p=(u,v,z)$ has square zero and
$A/\mathfrak p\cong k[[t]]$, so $\dim A=1$.
Since $\soc_A A\neq0$, one has $\operatorname{depth}_A A=0$.
Then $A$ is not Cohen--Macaulay and hence is not Gorenstein;
in particular, $\idim_A A=\infty$; see \cite[Chapter~3]{BH}.

The ideal $\mathfrak p$ is prime. With $K=k((t))$, localization gives
\[
 A_{\mathfrak p}\cong K[[u,v]]/(u,v)^2,
\]
since $t$ becomes invertible and $tz=0$ forces $z=0$.
The unique simple $A_{\mathfrak p}$-module
$T=A_{\mathfrak p}/(u,v)A_{\mathfrak p}$ satisfies
\[
 \Hom_{A_{\mathfrak p}}(T,A_{\mathfrak p})
 =\soc_{A_{\mathfrak p}}A_{\mathfrak p}=Ku\oplus Kv.
\]
Consequently, $\grade_{A_{\mathfrak p}}T=0$ and $h_{A_{\mathfrak p}}(T)\cong T^{\oplus2}$, which
is not simple. Then $A_{\mathfrak p}$ does not have the Iyama grade duality.
\end{example}

A module-finite algebra $A$ over $R$ 
is called locally Iwanaga--Gorenstein
if $A_{\mathfrak p}$ is Iwanaga--Gorenstein
for every $\mathfrak p\in\Spec R$.
Suppose that $A$ is locally Iwanaga--Gorenstein. 
Then for every $\mathfrak p\in \Spec R$ and every simple
$A_{\mathfrak p}$-module $S$, Lemma~\ref{lem:grade-fd} and
Lemma~\ref{lem:iwanaga} give
\[
\grade_{A_{\mathfrak p}}S
\leq
\fdim_{A_{\mathfrak p}}I_{A_{\mathfrak p}}(S)
\leq
\idim_{A_{\mathfrak p}^{\op}}A_{\mathfrak p}
<\infty.
\]

The following gives a local analog of \cite[Theorem~0.3]{KKMM} 
for Noether algebras.

\begin{theorem}\label{thm:fiberwise-tor-short}
Let $A$ be a module-finite algebra over a commutative Noetherian ring $R$. 
Then the following are equivalent.
\begin{enumerate}
\item $A$ is locally Iwanaga--Gorenstein and 
      admits a fiberwise Iyama grade duality;
\item for every $\mathfrak p\in \Spec R$ and every simple $A_{\mathfrak p}$-module $S$,
      \[\grade_{A_{\mathfrak p}}S=\fdim_{A_{\mathfrak p}}I_{A_{\mathfrak p}}(S)<\infty.\]
\end{enumerate}
\end{theorem}

\begin{proof}
(1)$\implies$(2). 
Fix $\mathfrak p\in\Spec R$ and a simple
$A_\mathfrak p$-module $S$. 
Let $I$ be an injective envelope of $S$.
Since $A$ is locally Iwanaga--Gorenstein,
Lemma~\ref{lem:iwanaga} implies that $d=\fdim_{A_\mathfrak p}I$ is finite.
Moreover, Lemma~\ref{lem:plocal-ptorsion} shows that $I$ is
$\mathfrak p$-torsion and hence $\mathfrak pR_{\mathfrak p}$-torsion
as an $R_{\mathfrak p}$-module.

Applying Corollary~\ref{cor:residue-tor-short} 
over the local ring $R_\mathfrak p$ gives
\[
\Tor_d^{A_\mathfrak p}(A_\mathfrak p/\rad A_\mathfrak p,I)\neq0.
\]
Since $A_\mathfrak p/\rad A_\mathfrak p$ is semisimple Artinian, there is a
simple $A_\mathfrak p^{\op}$-module $T$ such that
\[
\Tor_d^{A_\mathfrak p}(T,I)\neq0.
\]
By the fiberwise Iyama grade duality, there exists a simple
$A_\mathfrak p$-module $S'$ of finite grade $g$ such that
\[
T\cong \soc_{A_\mathfrak p^{\op}}
\Ext_{A_\mathfrak p}^{g}(S',A_\mathfrak p)
\subseteq\Ext_{A_\mathfrak p}^{g}(S',A_\mathfrak p).
\]
Lemma~\ref{lem:tor-grade} gives $d=g$ and
$\Hom_{A_\mathfrak p}(S',I)\neq0$. A nonzero map from $S'$
to the injective module $I$ is injective, and essentiality
of $S\subseteq I$ forces $S'\cong S$. Then
\[
\fdim_{A_\mathfrak p}I_{A_\mathfrak p}(S)=\grade_{A_\mathfrak p}S
\]
for every $(\mathfrak p,S)$. Corollary~\ref{cor:grade=fd} shows that $A$
satisfies the Auslander condition.

(2)$\implies $(1). 
Fix $\mathfrak p\in\Spec R$. 
By Corollaries~\ref{cor:grade=fd}
and~\ref{cor:localglobal}, $A_{\mathfrak p}$ satisfies
the Auslander condition.
Since $A_{\mathfrak p}$ has only finitely many
isomorphism classes of simple modules,
Lemma~\ref{lem:iwanaga} and condition~(2) give
\[
d=\idim_{A_{\mathfrak p}^{\op}}A_{\mathfrak p}
=\max_{S\in\Sim A_{\mathfrak p}}
  \fdim_{A_{\mathfrak p}}I_{A_{\mathfrak p}}(S)
<\infty.
\]
It follows from Theorem~\ref{thm:main} that 
\[
\idim_{A_\mathfrak p}A_\mathfrak p=
\idim_{A_\mathfrak p^{\op}}A_\mathfrak p=d<\infty.
\]
Then $A_\mathfrak p$ is Iwanaga--Gorenstein.
Since $A_\mathfrak p$ is $(d+1)$-Gorenstein, 
Corollary~\ref{cor:grade=fd} and
Lemma~\ref{lem:iwanaga} give upper bounds
\[
\grade_{A_\mathfrak p}S<d+1 \text{ and } 
\grade_{A^{\op}_\mathfrak p}T<d+1
\]
for every simple $A_\mathfrak p$-module $S$ and 
every simple $A^{\op}_\mathfrak p$-module $T$. 
Then 
\[\Sim_{<{d+1}}(A_\mathfrak p)=\Sim A_\mathfrak p  \text{ and } 
\Sim_{<{d+1}}(A^{\op}_\mathfrak p)=\Sim A^{\op}_\mathfrak p.\]
Lemma~\ref{lem:iyama} now shows that
$h_{A_\mathfrak p}$ is a bijection between all simple
$A_\mathfrak p$-modules and all simple $A_\mathfrak p^{\op}$-modules.
Hence, $A$ has the fiberwise Iyama grade duality.
\end{proof}

Let $A$ be a module-finite algebra over a commutative Noetherian ring $R$. 
Self-injective dimension can be computed by central localization.
Indeed, Ext groups between finite $A$-modules commute with central
localization, and their vanishing can be checked at maximal ideals
of $R$. Then
\[
\idim_A A
=
\sup_{\mathfrak p \in\Spec R}
\idim_{A_{\mathfrak p}}A_{\mathfrak p}
=
\sup_{\mathfrak m\in\Max R}
\idim_{A_{\mathfrak m}}A_{\mathfrak m}.
\]

We have the following.

\begin{corollary}\label{cor:tor-one-sided}
Let $A$ be a module-finite algebra over a commutative Noetherian ring $R$. 
Then the following are equivalent.
  \begin{enumerate} 
    \item $A$ is Auslander--Gorenstein;
    \item $A$ is Iwanaga--Gorenstein and 
    $A$ has the fiberwise Iyama grade duality.
  \end{enumerate}
\end{corollary}

\begin{proof}
Assume (1). 
Every central localization $A_{\mathfrak p}$ is Auslander--Gorenstein.
Then Theorem~\ref{thm:fiberwise-tor-short} gives
the fiberwise Iyama grade duality.

Conversely, assume (2).
Central localization gives
\[
\idim_{A_{\mathfrak p}}A_{\mathfrak p}
\leq \idim_{A}A<\infty
\]
for every $\mathfrak p\in\Spec R$.
By Theorem~\ref{thm:fiberwise-tor-short}, $A$ satisfies
the Auslander condition.
Theorem~\ref{thm:main} then shows that $A$ is
Auslander--Gorenstein.
\end{proof}

\section{Cohen--Macaulay algebras and central regular reductions}
\label{sec:cm-reductions}

In this section, we use central regular sequences to reduce
Cohen--Macaulay algebras to Artin algebras and prove the
Cohen--Macaulay assertions in Corollary~\ref{cor:intro-ac}
and Theorem~\ref{thm:intro-third}.

Let $A$ be a two-sided Noetherian ring 
and $x\in Z(A)$. For an $A$-module $M$, put 
\[
\Gamma_{xA}(M)=\bigcup_{r\geq1}(0:_M x^r).
\]
The module $M$ is called $x$-torsion if $\Gamma_{xA}(M)=M$,
and $x$-torsionfree if $\Gamma_{xA}(M)=0$.

The full Auslander condition criterion involving $\bar A$ and
$A_x$ is the central ungraded case of
\cite[Theorem~4.4]{ZSL}.
The finite $n$-Gorenstein statements and the radical case
needed below are proved here.

\begin{proposition}\label{prop:central-reduction}
Let $A$ be a two-sided Noetherian ring, let $x\in J(A)$ be a
central regular element, and let $n\geq1$.
Write $\bar A=A/xA$ and $A_x=A[x^{-1}]$.
\begin{enumerate}
\item $A$ is $n$-Gorenstein if and only if
      $A_x$ is $n$-Gorenstein and $\bar A$ is $(n-1)$-Gorenstein.
\item If $\bar A$ is $n$-Gorenstein,
      then $A$ is $n$-Gorenstein.
\item $\idim_A A=\idim_{\bar A}\bar A+1$. 
In particular, $A$ is
Iwanaga--Gorenstein if and only if $\bar A$ is.
\end{enumerate}
\end{proposition}

\begin{proof}
If $I$ is an indecomposable $x$-torsion injective $A$-module, then
$J=(0:_I x)$ is injective over $\bar A$ by adjunction.
Every nonzero submodule of $I$ meets $J$, so $J$ is essential in $I$.
Since an indecomposable injective over a Noetherian ring is uniform,
$J$ is indecomposable. We have
\[\begin{aligned}
 \occ_A I &=\occ_{\bar A}J+1\\
 \fdim_A I &=\fdim_{\bar A}J+1.
\end{aligned}\]
Every indecomposable injective over $\bar A$ arises this way.
Indeed, let $J$ be an indecomposable injective $\bar A$-module
and take $I=I_A(J)$. The Artin--Rees lemma for the central ideal
$xA$ shows that $I$ remains $x$-torsion.
The inclusion $J\subseteq(0:_I x)$ splits over $\bar A$.
Since $J$ is essential in $I$, the complementary summand is zero,
so $J=(0:_I x)$. Moreover, $I$ is indecomposable. 
The occurrence shift follows by applying $\Hom_A(\bar A,-)$ to a minimal
injective resolution of $A$.
The degree-zero term vanishes, and the resulting complex,
shifted by one, is a minimal injective coresolution of $\bar A$.

We prove the flat-dimension equality in two steps.
For every $\bar{A}^{\op}$-module $V$, the change-of-rings
isomorphism gives
\[
\Tor_{j+1}^A(V,I)\cong\Tor_j^{\bar A}(V,J)
\]
for $j\geq 0$. Consequently,
$\fdim_A I\geq\fdim_{\bar A}J+1$.

For the reverse inequality, assume that
$b=\fdim_{\bar A}J$ is finite.
Let $N$ be an $A^\op$-module and let $t>b+1$.
Put $K=(0:_N x)$ and $Q=N/Nx$.
Since $K$ and $Q$ are $\bar A^{\op}$-modules, the above
isomorphism gives
\[
\Tor_t^A(K,I)=0 \text{ and }
\Tor_{t+1}^A(Q,I)=0.
\]
The long exact sequences associated to
\[\begin{aligned}
&0\longrightarrow K\longrightarrow N
\xrightarrow{x}Nx\longrightarrow0,\\
&0\longrightarrow Nx\longrightarrow N
\longrightarrow Q\longrightarrow0
\end{aligned}\]
therefore give embeddings
\[
\Tor_t^A(N,I)\longrightarrow\Tor_t^A(Nx,I)
\longrightarrow\Tor_t^A(N,I).
\]
Their composite is multiplication by $x$.
Since $\Tor_t^A(N,I)$ is $x$-torsion, this forces it to vanish.
Then $\fdim_A I\leq b+1$, proving the equality.

By \cite[Lemma~11.1]{PPT} and \cite[Example~3.13]{VY},
the $x$-torsion submodule $\Gamma_{xA}(I)$ 
is injective for every injective $A$-module $I$.
Then the exact sequence
\[
0\longrightarrow\Gamma_{xA}(I)\longrightarrow I
\longrightarrow I/\Gamma_{xA}(I)\longrightarrow0
\]
splits. Every indecomposable injective $A$-module is 
either $x$-torsion or $x$-torsionfree.
An $x$-torsionfree injective $A$-module is naturally an
$A_x$-module, since multiplication by $x$ is also surjective.
Restriction of scalars identifies the indecomposable injective
$A_x$-modules with the $x$-torsionfree indecomposable injective
$A$-modules.
Central localization preserves minimal injective coresolutions,
and restriction of scalars preserves injectivity here.
Then occurrence degree and flat dimension are unchanged.
Applying Lemma~\ref{lem:n-gor-criteria} to these two classes
proves (1).

For (2), assume that $\bar A$ is $n$-Gorenstein.
By (1), it suffices to prove that $A_x$ is $n$-Gorenstein.
Let $X$ be a finite $A_x$-module. Choose a finite $A$-submodule
$M$ of $X$ such that $M_x=X$. Then $M$ is $x$-torsionfree.
Fix $1\leq i\leq n$, and put
\[
E=\Ext_A^i(M,A) \text{ and }
F=E/\Gamma_{xA}(E).
\]
Then $E_x$ is isomorphic to $\Ext_{A_x}^i(X,A_x)$,
and $F$ is $x$-torsionfree.
We identify $F$ with its image in $F_x=E_x$.
Let $V$ be an $A_x^{\op}$-submodule of $E_x$, and put $L=V\cap F$.
Then $L_x=V$, and $F/L$ is $x$-torsionfree. Hence,  
$L/xL$ embeds in $F/xF$. Moreover, $F/xF$ is a quotient of
$E/xE$, and the exact sequence 
for multiplication by $x$ gives 
\[
E/xE\rightarrowtail
\Ext_{\bar A}^i(M/xM,\bar A).
\]
By Lemma~\ref{lem:n-gor-criteria},
$\Ext_{\bar A}^i(M/xM,\bar A)$ has strong grade at least $i$.
The strong-grade classes are Serre classes, so the same holds
for its subquotient $L/xL$. In particular,
\[\grade_{\bar A^{\op}}(L/xL)\geq i.\]

By Rees' isomorphisms,
$\Ext_{A^{\op}}^j(L/xL,A)=0$ for $0\leq j\leq i$.
The exact sequence for multiplication by $x$ on $L$ shows that
multiplication by $x$ is surjective on $\Ext_{A^{\op}}^j(L,A)$
for $j<i$. Since these modules are finite and $x\in J(A)$,
Nakayama's lemma gives $\Ext_{A^{\op}}^j(L,A)=0$ for $j<i$.
Localizing shows that $\grade_{A_x^{\op}}V\geq i$.
Lemma \ref{lem:n-gor-criteria} implies that $A_x$ is $n$-Gorenstein.

For (3), we prove $\idim_AA=\idim_{\bar A}{\bar A}+1$.
Recall Rees' isomorphism
\[
 \Ext_A^{j+1}(M,A)\cong
 \Ext_{\bar A}^{j}(M,\bar A)
\]
for any $j\geq0$ and any finite $\bar A$-module $M$.
This gives
\[
 \idim_A A\geq\idim_{\bar A}\bar A+1.
\]
If $b=\idim_{\bar A}\bar A<\infty$ and $L$ is a finite
$x$-torsionfree $A$-module, the exact sequence
\[
0\longrightarrow L\xrightarrow{x}L
\longrightarrow L/xL\longrightarrow0
\]
makes $x$ surjective on $\Ext_A^{b+2}(L,A)$.
This finite $A^{\op}$-module vanishes by Nakayama's lemma.
The $x$-torsion submodule of any finite $A$-module is killed
by a power of $x$ and has a finite filtration by $\bar A$-modules.
Rees' isomorphisms give the same vanishing for the torsion
submodule. The exact sequence separating the torsion submodule
from the torsionfree quotient then gives
$\Ext_A^{b+2}(M,A)=0$ for every finite $A$-module $M$.
Then $\idim_A A\leq b+1$, proving the equality.
Applying the same argument to $A^{\op}$ gives the
Iwanaga--Gorenstein equivalence.
\end{proof}

Let $A$ be a module-finite algebra over a commutative Noetherian ring $R$.
Recall that $A$ is {\em Cohen--Macaulay} over $R$ if for every
$\mathfrak p\in\operatorname{Supp}_R A$,
\[
 \operatorname{depth}_{R_{\mathfrak p}}A_{\mathfrak p}
 =\dim_{R_{\mathfrak p}}A_{\mathfrak p}.
\]

Until otherwise stated, 
we assume that $(R,\mathfrak m)$ is local and 
$A$ is a nonzero module-finite algebra over $R$.
Let $d$ be the Krull dimension of $A$ as an $R$-module.

\begin{theorem}\label{thm:cm-parameters}
Assume that $A$ is Cohen--Macaulay.
Choose any system of parameters $x_1,x_2,\ldots,x_d\in\mathfrak m$ 
for $A$. 
For $0\leq i\leq d$, set 
\[A_i=A/(x_1,x_2,\ldots,x_i)A \]
 Then we have 
\[\dim_RA_i=\operatorname{depth}_RA_i=d-i.
\]
Moreover, $A_d$ is an Artin algebra.
\end{theorem}

\begin{proof}
A system of parameters on a finite Cohen--Macaulay module is a regular
sequence, and its successive quotients are Cohen--Macaulay with dimension
and depth reduced by one; see \cite[Chapter~2]{BH}.
This gives the equality of dimension and depth.
The quotients are nonzero by Nakayama's lemma. Since $R$ acts centrally,
each $x_i$ is regular on both sides of $A_{i-1}$.

The finite $R$-module $A_d$ has dimension zero and hence finite length.
Then $R/\operatorname{ann}_R A_d$ is Artinian, and $A_d$ is finite
over this ring. Therefore, $A_d$ is an Artin algebra.
\end{proof}

\begin{corollary}\label{cor:cm-artin}
Assume that $A$ is Cohen--Macaulay over $R$. The following hold.
\begin{enumerate}
\item $A$ satisfies the Auslander condition if and only if 
    $A_d$ does so.
\item $A$ is Iwanaga--Gorenstein if and only if 
    $A_d$ is Iwanaga--Gorenstein.    
\item $A$ is Auslander--Gorenstein if and only if 
    $A_d$ is Auslander--Gorenstein.
\end{enumerate}
\end{corollary}

\begin{proof}
For each $0\leq i\leq d$, let $S$ be a simple $A_i$-module.
Since $S$ is a quotient of $A_i$, it is finite over $R$.
The central action of $R$ makes $\mathfrak mS$ an $A_i$-submodule
of $S$, so $\mathfrak mS$ is either zero or $S$.
The latter is excluded by Nakayama's lemma. Then
$\mathfrak mA_i$ annihilates every simple $A_i$-module, and hence
$\mathfrak mA_i\subseteq J(A_i)$.

For $1\leq i\leq d$, the image of $x_i\in\mathfrak m$ in
$A_{i-1}$ belongs to $J(A_{i-1})$.
It is regular by Theorem~\ref{thm:cm-parameters}, and its quotient
is $A_i$. Proposition~\ref{prop:central-reduction} can 
be applied to $A_{i-1}\to A_i$ for $1\leq i\leq d$.
For the Auslander condition, apply parts (1) and (2) for every $n$.
Part (3) gives the Iwanaga--Gorenstein equivalence, and combining
the two equivalences gives the Auslander--Gorenstein assertion.
\end{proof}

The same reduction gives formulas for the grades of simple modules and
the flat dimensions of their injective envelopes.

\begin{lemma}\label{lem:cm-simple-shifts}
Assume that $A$ is Cohen--Macaulay over $R$. 
 For every simple $A$-module $S$, we have
\[\begin{aligned}
 \grade_A S&=\grade_{A_d}S+d,\\
 \fdim_A I_A(S)&=\fdim_{A_d}I_{A_d}(S)+d.
\end{aligned}\]
Moreover, $A$ has the Iyama grade duality if and only if $A_d$ does.
\end{lemma}

\begin{proof}
Since $\mathfrak mA\subseteq\rad A$, every simple $A$-module is
annihilated by all the parameters. The same argument applies to $A^{\op}$, giving
the asserted identifications of simple modules.
Successive applications of Rees' isomorphisms give
\[
 \Ext_A^{d+j}(S,A)\cong\Ext_{A_d}^{j}(S,A_d) 
\]
for $j\geq0$ and $\Ext_A^j(S,A)=0$ for $j<d$.
This proves the grade equality.

Put $I_i=I_{A_i}(S)$. By Lemma~\ref{lem:plocal-ptorsion}, $I_{i-1}$
is $\mathfrak m$-torsion. Its submodule annihilated by $x_i$ is
injective over $A_i$ by adjunction and contains $S$ essentially.
Consequently,
\[
 I_i\cong(0:_{I_{i-1}}x_i).
\]
The flat-dimension shift in the proof of
Proposition~\ref{prop:central-reduction} yields
\[
 \fdim_{A_{i-1}}I_{i-1}=\fdim_{A_i}I_i+1.
\]
Iteration gives the flat-dimension equality. 

For $S$ of finite grade, the Ext isomorphism in its first nonzero
degree gives
\[
 \soc_{A^{\op}}\Ext_A^{\grade_A S}(S,A)
 \cong
 \soc_{A_d^{\op}}\Ext_{A_d}^{\grade_{A_d}S}(S,A_d).
\]
The two socles agree because the $A^{\op}$-module structure factors through $A_d^{\op}$.
Then the grade maps agree under the identifications of simple modules.
The grade equality also identifies the simple modules of finite grade,
proving the last assertion.
\end{proof}

We now allow $R$ to be a commutative Noetherian ring.
The following observation relates simple modules over $A$ to those
over its localizations at maximal ideals.

\begin{lemma}\label{lem:cm-maximal-simples}
Let $A$ be a module-finite $R$-algebra. Then
\[
 \Sim A=\coprod_{\mathfrak m\in\Max R}\Sim A_{\mathfrak m}.
\]
For a simple module $S$ belonging to the component at $\mathfrak m$,
\[\begin{aligned}
 \grade_A S &=\grade_{A_{\mathfrak m}}S,\\
 \fdim_A I_A(S) &=\fdim_{A_{\mathfrak m}}I_{A_{\mathfrak m}}(S).
\end{aligned}\]
Furthermore, $A$ has the Iyama grade duality if and only if
$A_{\mathfrak m}$ has it for every $\mathfrak m\in\Max R$.
\end{lemma}

\begin{proof}
Since $A$ is module-finite over $R$, the annihilator in $R$ of a
simple $A$-module is a maximal ideal $\mathfrak m$. The simple modules
with this annihilator are precisely the simple $A/\mathfrak mA$-modules,
and hence the simple $A_{\mathfrak m}$-modules.
Each $\Ext_A^j(S,A)$ is annihilated by $\mathfrak m$.
Central localization yields
\[
 \Ext_A^j(S,A)\cong
 \Ext_{A_{\mathfrak m}}^j(S,A_{\mathfrak m}).
\]
This gives the grade equality and the socle identification. 
Proposition~\ref{prop:dictionary} gives the flat-dimension equality. 
The grade maps respect the decompositions on both sides, proving the last assertion.
\end{proof}

We can now prove the Cohen--Macaulay assertion in
Corollary~\ref{cor:intro-ac}.

\begin{theorem}\label{thm:cm-ac-simples}
Let $A$ be a module-finite Cohen--Macaulay algebra over a commutative
Noetherian ring $R$. Then $A$ satisfies the Auslander condition if and
only if for every simple $A$-module $S$, one has
$\fdim_A I_A(S)=\grade_A S$.
\end{theorem}

\begin{proof}
If $A$ satisfies the Auslander condition, the equality follows from
Corollary~\ref{cor:grade=fd} and Lemma~\ref{lem:cm-maximal-simples}.

Conversely, assume the equality for all simple $A$-modules. By
Lemma~\ref{lem:cm-maximal-simples}, 
the same equality holds over each $A_{\mathfrak m}$. 
By Corollary~\ref{cor:localglobal}, it suffices to prove the assertion
when $R$ is local. We may therefore assume that 
$R$ is local and $A$ is nonzero.

Choose a parameter sequence as in Theorem~\ref{thm:cm-parameters}.
Lemma~\ref{lem:cm-simple-shifts} gives
\[
 \fdim_{A_d}I_{A_d}(S)=\grade_{A_d}S
\]
for every simple $A_d$-module $S$. 
The algebra $A_d$ is Artin, so every indecomposable injective is the
injective envelope of a simple module. Lemmas~\ref{lem:grade-fd}
and~\ref{lem:n-gor-criteria} show that $A_d$ is $n$-Gorenstein for every
$n$. Then $A_d$ satisfies the Auslander condition, and
Corollary~\ref{cor:cm-artin} gives the same assertion for $A$.
\end{proof}

Finally, we apply the reduction to the Iyama grade duality.

\begin{theorem}\label{thm:cm-iyama}
Let $A$ be a module-finite Cohen--Macaulay algebra over a commutative
Noetherian ring $R$. The following are equivalent.
\begin{enumerate}
\item $A$ is Auslander--Gorenstein.
\item $A$ is Iwanaga--Gorenstein and has the Iyama grade duality.
\end{enumerate}
\end{theorem}

\begin{proof}
The implication from (1) to (2) follows from
Corollary~\ref{cor:tor-one-sided} and
Lemma~\ref{lem:cm-maximal-simples}.

Assume (2). By localization and Lemma~\ref{lem:cm-maximal-simples},
$A_{\mathfrak m}$ satisfies (2) for every maximal ideal
$\mathfrak m$ of $R$.
Since $A$ is Iwanaga--Gorenstein, Corollary~\ref{cor:localglobal}
reduces the proof to verifying the Auslander condition locally.
We may therefore assume that $R$ is local with maximal ideal
$\mathfrak m$ and $A$ is nonzero.

Put $d=\dim_R A$, choose a system of parameters
$x_1,\ldots,x_d$ on $A$, and set
\[
A_d=A/(x_1,\ldots,x_d)A.
\]
Corollary~\ref{cor:cm-artin} shows that $A_d$ is
Iwanaga--Gorenstein, and Lemma~\ref{lem:cm-simple-shifts}
gives its Iyama grade duality.
Since $\operatorname{Supp}_R A_d=\{\mathfrak m\}$ and
$(A_d)_{\mathfrak m}\cong A_d$, this duality is fiberwise.
By Corollary~\ref{cor:tor-one-sided}, $A_d$ is
Auslander--Gorenstein.
Corollary~\ref{cor:cm-artin} then shows that $A$ satisfies
the Auslander condition.
This proves that $A$ is Auslander--Gorenstein.
\end{proof}

We conclude with an explicit example of the grade maps.

\begin{example}\label{ex:cm-order}
Let $k$ be a field, $R=k[[t]]$, and $K=k((t))$. Put
\[
 A=\begin{pmatrix}R&R\\tR&R\end{pmatrix}\subseteq M_2(K).
\]
This is the standard non-maximal hereditary order; see
\cite[Corollary~21.1.5]{Voight}.
The algebra $A$ is free over $R$ with basis $e_1,e_2,u,v$, where
$u=e_{12}$ and $v=te_{21}$. Then $A$ is Cohen--Macaulay and $t$ is
a central regular parameter. 

In $B=A/tA$ the images of $u,v$ are
nonzero and satisfy
\[
\begin{aligned}
 e_1u&=u=ue_2, & e_2v&=v=ve_1,\\
 uv&=0, & vu&=0.
\end{aligned}
\]
Hence, $\rad B=ku\oplus kv$, with radical square zero. Write
\[
\begin{aligned}
 S_i&=Be_i/(\rad B)e_i,\\
 T_i&=e_iB/e_i\rad B.
\end{aligned}
\]
We identify $S_i$ and $T_i$ with the corresponding simple
$A$-modules and $A^{\op}$-modules, respectively. The projectives $Be_1,Be_2$ have socles $S_2,S_1$,
respectively. They are the indecomposable
injectives. Indeed,
\[
 Be_1\cong D(e_2B),\;Be_2\cong D(e_1B),
\]
where  $D=\Hom_k(-,k)$.
It follows that $B$ is self-injective. Furthermore,
\[
\begin{aligned}
 \Hom_B(S_1,B)&=ku\cong T_2,\\
 \Hom_B(S_2,B)&=kv\cong T_1.
\end{aligned}
\]
Then $A$ is Auslander--Gorenstein with self-injective dimension one and 
\[\begin{aligned}
 &\grade_A S_i=\fdim_A I_A(S_i)=1,\\
 &h_A(S_1)=T_2,\;h_A(S_2)=T_1.
\end{aligned}\]
On the other hand, 
\[A_{(0)}=A\otimes_R K=M_2(K).\]
The generic localization has a single simple module of grade zero,
whereas $A$ has two simple modules of grade one. The fiberwise Iyama
grade duality is therefore a bijection between singleton sets at
$(0)$ and exchanges the two labels at $(t)$.
\end{example}

Theorem~\ref{thm:cm-iyama} raises the following question.

\begin{question}
Let $A$ be a module-finite algebra over a commutative
Noetherian ring $R$. If $A$ is Iwanaga--Gorenstein and has
the Iyama grade duality, must $A$ be Auslander--Gorenstein?
\end{question}

\section*{Declaration of AI use}

The author used ChatGPT (GPT5.6 and GPT6) to assist with language editing, 
exposition, and preliminary checks of mathematical arguments. 
All suggestions were reviewed and verified by the author, who takes full
responsibility for the final manuscript.

\section*{Acknowledgements}

This work is partially supported by the National Natural Science Foundation
of China (Nos.~12571036 and 11801141).

\end{document}